\documentclass[12pt]{amsart}
\usepackage{graphicx} 
\usepackage{amsmath,amssymb,amsthm} 

\usepackage{tikz-cd}
\usetikzlibrary{arrows.meta,decorations.pathmorphing}
\usepackage{mathtools}
\usepackage{enumitem}
\usepackage{leftindex}
\usepackage{parskip}
\usepackage{hyperref}

\newtheorem{theorem}{Theorem}[section]
\newtheorem{prop}[theorem]{Proposition}
\newtheorem{corollary}[theorem]{Corollary}
\newtheorem{lemma}[theorem]{Lemma}
\newtheorem{example}[theorem]{Example}
\newtheorem{conjecture}[theorem]{Conjecture}
\newtheorem{remark}[theorem]{Remark}

\newcommand{\Tor}{\text{Tor}}
\newcommand{\ov}[2]{\overline{#1}/\overline{#2}}
\DeclareMathOperator{\HF}{HF}
\DeclareMathOperator{\HS}{H}

\title[ACI's generated by powers of general linear forms]{Betti numbers of ideals generated by $n+1$ powers of general linear forms}
\author{Eric Dannetun}
\address{Eric Dannetun, Department of Mathematics, Stockholm University, 106 91 Stockholm,
Sweden}
\email{eric.dannetun@math.su.se}

\begin{document}
\begin{abstract}
    We study ideals generated by $n+1$ powers of general linear forms in $R= k[x_1,\dots,x_n]$. By generalizing the ideas in a recent paper of Diethorn et al., we determine the Betti numbers of such ideals when at least one generator is a square. It follows that all such ideals are level. As a consequence, we show that a generic ideal in $R$ generated by $n+1$ forms, with at least one quadric generator, is level. We also determine the Betti numbers of the Artinian Gorenstein algebras linked to these almost complete intersections. By describing the dual generators of these algebras, we obtain a family of forms, including the elementary symmetric polynomials, whose annihilator ideals have the strong Lefschetz property. Finally, we give explicit generators for the annihilator ideal of any elementary symmetric polynomial.
\end{abstract}
\maketitle
\tableofcontents

\section{Introduction}
Let $R={k}[x_1,\dots,x_n]$, with $k$ of characteristic 0, and consider the ideal $I = (\ell_1^{d_1},\dots,\ell_{n+1}^{d_{n+1}})$, where the $\ell_i$ are general linear forms. After a linear change of variables $I$ can be assumed to be of the form  $I=(x_1^{d_1},\dots,x_n^{d_n},\ell^{d_{n+1}})$  where $\ell=x_1+\dots +x_n$. The study of ideals generated by general linear forms, not necessarily almost complete intersections, has a a rich history. We list some notable contributions: \cite{Fro-Holl}, \cite{EI-fatpoints}, \cite{SS-linear-forms-3var}, \cite{fat-point-scheme}, \cite{Nage-Trok-interpolation}, and \cite{BL-classification}. Moreover, the case of almost complete intersections has recently received much attention from many authors from different perspectives, see for example \cite{booth2025weaklefschetzpropertyideals}, \cite{2024grobnerbasesresolutionslefschetz}, \cite{diethorn2025studyquadraticcompleteintersection} and \cite{2025grobnerbasispowersgeneral}.

In this paper we determine the Betti numbers for any ideal generated by $n+1$ powers of general linear forms where at least one generator is a square. This generalizes recent work by Diethorn et al. \cite{diethorn2025studyquadraticcompleteintersection} who study the special case of the ideal $(x_1^2,\dots,x_n^2,\ell^2)$, and our approach is heavily inspired by this work.
 As we will discuss later there are some situations where the Betti numbers were already understood. We determine the Betti numbers for the remaining cases by showing that these can be written as a sum of two Betti numbers coming from such understood ideals.

  As a consequence, we prove that a generic ideal generated by $n+1$ forms, where at least one of the generators is a quadric, is level. 

  Let $J = (x_1^{d_1},\dots,x_{n-1}^{d_{n-1}},x_n^{d_n})$. Associated to our ideal $I=J+(\ell^{d_n+1})$ we also have an Artinian Gorenstein algebra $R/G$, where $G=J\colon (\ell^{d_{n+1}})$. The rest of the paper is mainly devoted to this algebra. In a similar fashion to $R/I$ we determine the Betti numbers of $R/G$ when one of the generators of $J$ is a square. We show that $R/G$, for arbitrary $d_1,\dots,d_{n+1}$, satisfy the strong Lefschetz property and moreover use this together with a description of the Macaulay dual generators to obtain a new family of forms which are dual generators of  Artinian Gorenstein algebras with the strong Lefschetz property. Lastly we find a set of generators of $G$ when $d_1=\dots=d_n =2$, extending the work of \cite{diethorn2025studyquadraticcompleteintersection}.

\subsection{Background and motivation}

Ideals generated by powers of linear forms appear naturally in the study of Lefschetz properties, and by the Macaulay inverse system they are linked to fat point ideals and hence to questions of polynomial interpolation. Notably, by work of R. Stanley and independently by J. Watanabe (see Theorem \ref{thm: monom-SLP}), the Hilbert series of $I=(x_1^{d_1},\dots,x_n^{d_n},\ell^{d_{n+1}})$ is minimal among all ideals generated by $n+1$ forms of degrees $d_1,\dots,d_{n+1}$. Moreover, it is precisely this that constitutes the known proof of Fröberg's conjecture (see \cite{Frobergs-formodan}) for ideals generated by $n+1$ forms.

A natural continuation of this work, both for generic ideals and ideals generated by powers of general linear forms with $n+1$ generators, is to determine the minimal free resolution, i.e. the Betti numbers. The study of Betti numbers for generic ideals was initiated in  \cite{res-n+1-general-forms} where, among other results, the Betti numbers, for the case of 4 generators, are determined when $n=3$. The work was later improved in \cite{relatively-compressed} and more results are given in \cite{ubiquity}, \cite{Pardue-Richert} and \cite{Diem}. Roughly these results say that the minimal resolution of a generic ideal (where Fröberg's conjecture is assume to hold), will look like the corresponding Koszul complex in all but the two highest degrees of each free module. In some special cases it will coincide with the Koszul complex in all but the highest degree, and in this case the last Betti numbers are forced by the Hilbert series, and thus, only in these special cases, are the Betti numbers understood fully.  For $n+1$ forms, this argument also holds for ideals generated by powers of general linear forms.

 In light of this, and as they are the main examples of ideals with $n+1$ generators known to have minimal Hilbert series, it is also natural to approach the Betti numbers of generic ideals by studying powers of general linear forms. 
 
\section{Preliminaries}
Let $R=\bigoplus_{i\ge 0} R_i$ be a Noetherian standard graded ring with $R_0=k$ a field of characteristic zero. We will generally consider $R$ as the polynomial ring in $n$ variables, $k[x_1,\dots,x_n]$, or as a graded quotient of this ring.  The unique graded maximal ideal of $R$ will be denoted $\mathfrak{m} = \bigoplus_{i>0}R_i$. We begin by recalling some standard definitions and notation, as well as some relevant results. Much of the standard definitions and notation can be found in \cite{Peeva-Graded-Syz}.

The Hilbert function, $\HF_M\colon \mathbb{N}\to \mathbb{N}$, of a finitely generated graded $R$-module $M$, is defined as
$$
\HF_M(i) := \dim_k (M_i),
$$
and the Hilbert series, $\HS_M(T)\in \mathbb{Z}[[T]]$, as
$$
\HS_M(T) := \sum_{i=0}^\infty \HF_M(i)T^i.
$$
If $M$ has Krull dimension $d$ we can write $\HS_M(T) = \frac{h_m(T)}{(1-T)^d}$ and the multiplicity of $M$, denoted $e(M)$, is defined as $h_M(1)$.

A graded free resolution 
$$
F_\bullet \colon \quad F_n \xrightarrow{\phi_n} F_{n-1} \xrightarrow{\phi_{n-1}} \dots  \xrightarrow{\phi_3} F_2 \xrightarrow{\phi_{2}} F_1 \xrightarrow{\phi_1} F_0 \
$$
of $M$ is said to be minimal if $\phi_i(F_{i+1}) \subseteq \mathfrak{m}F_i$. A minimal resolution is unique up to a graded isomorphism, and writing $F_i = \bigoplus_{j}R(-j)^{\beta_{i,j}^R(M)}$ we call $\beta_{i,j}^R(M)$ the Betti numbers of $M$ (over $R$). It follows that
$$
\beta_{i,j}^R(M) = \dim_k \Tor_{i}^R(M,k)_j.
$$ The Betti numbers are often displayed in a table called a Betti table, in which the $(i,j)$:th entry is $\beta_{i,i+j}(M)$. See the figure below.

\begin{figure}[h]
$$
 \begin{matrix}
        & 0 & 1 & \dots & i & \dots & n \\
       0 :& \beta_{0,0} & \beta_{1,1} & \dots &\beta_{i,i}  & \dots  &   \beta_{n,n}\\
       1 :& \beta_{0,1} & \beta_{1,2} & \dots &   \beta_{i,i+1}  &  \dots&   \beta_{n,n+1}\\
       \vdots & \vdots & \vdots & \ddots &\ddots &\ddots & \vdots\\
       j:& \beta_{0,j} & \beta_{1,j+1} & \dots &   \beta_{i,i+j}  &  \dots&   \beta_{n,n+j}\\
         \vdots & \vdots & \vdots & \ddots & \ddots& \ddots& \vdots\\
         s: & \beta_{0,s} & \beta_{1,s+1} & \dots &   \beta_{i,i+s}  &  \dots&   \beta_{n,n+s}\\
       \end{matrix}.
$$

\end{figure}

For a graded free complex $C_\bullet$, which is not necessarily exact, we have $C_i = \bigoplus_j R(-j)^{\beta_{i,j}}$ and we also call $\beta_{i,j}$ the Betti numbers of $C_\bullet$.

For a sequence of elements $I = (f_1,\dots,f_r)$ we define the Koszul complex $K_\bullet(I)$ of this sequence by $(K_\bullet(I))_i = \bigwedge^i(R^r)$ with differential $d_i\colon K_i(I)\to K_{i-1}(I)$ given by
$$
d_i(e_{a_{1}}\wedge\dots \wedge e_{a_i}) = \sum_{j=1}^i(-1)^{j+1} f_{a_j} (e_{a_1}\wedge \dots \wedge\widehat{e_{a_j}} \wedge \dots \wedge e_{a_i})
$$ for any $a_1 < \dots < a_i \le r$, where $e_j$ denotes the $j$:th standard-basis element of $R^r$.

For any graded $R$-module $M$ and natural number $D$ we let $M_{(D)}$ denote the submodule generated by the elements of degree $D$ and lower. For a chain complex $(C_\bullet,d)$ satisfying $d(C_r) \subseteq \mathfrak{m}C_{r-1}$  we denote $C_\bullet^{(D)}$ the complex with $(C_\bullet^{(D)})_i = (C_{i})_{(i+D)}$ with the same (restricted) differential $d$ as $C_\bullet$. Note that the condition $d(C_r) \subseteq \mathfrak{m}C_{r-1}$ is sufficient for $C_\bullet^{(D)}$ to be a well defined chain complex. 

\begin{remark}
    For two chain complexes of free $R$-modules $C_\bullet$ and $G_\bullet$, the statement $C_\bullet^{(D)} \cong G_\bullet^{(D)}$ implies that the Betti tables of $C_\bullet $ and $G_\bullet$ are equal in rows $0$ to $D$.
\end{remark}

Related to these defintions we have the following useful lemma.
\begin{lemma}[Lemma 3.4 in \cite{Pardue-Richert}]\label{lemma: pr-truncated-resolution}
Let $\phi\colon M\to N$ be a module homomorphism inducing an isomorphism $M_{(\tau)}\cong N_{(\tau)}.$ If $F_\bullet \to M $ and $G_\bullet \to N$ are minimal free resolutions of $M$ and $N$, then any lift of $\phi$ to a morphism $F_\bullet \to G_\bullet$ restricts to an isomorphism $F_\bullet^{(\tau)} \cong G_\bullet^{(\tau)}$.
\end{lemma}

\begin{remark}
    We want to address the fact that there is a published errata \cite{Pardue-Richert-Errata} to the paper \cite{Pardue-Richert} which fixes some errors in the original proofs. This does not affect the validity of the results in the original paper.
\end{remark}

We say that a graded Artinian $k$-algebra $A$ has the strong Lefschetz property (SLP) if there exists a linear form $l \in A_1$ such that, for all integers $i,j$, the multiplication map
$$
A_i \xrightarrow{\cdot\ell^j} A_{i+j}
$$ has maximal rank, i.e. is either injective or surjective. $A$ is said to have the weak Lefschetz property (WLP) if the same multiplication map has maximal rank for $j=1$. An element for which the multiplication has maximal rank is called a weak, respectively strong, Lefschetz element.

The main result in the theory of Lefschetz properties, which also started the area, is the following. It was originally proved by Stanley \cite{Stanley}, later independently by Watanabe \cite{Watanabe}, and since then by many other authors.

\begin{theorem}\label{thm: monom-SLP}
Let $R=k[x_1,\dots,x_n]$ with $\text{char} \  k = 0$. Then any monomial complete intersection $I = (x_1^{d_1},\dots,x_n^{d_n})$ satisfies the SLP.
\end{theorem}

Similar to the Lefschetz properties we have the related notion of maximal-rank elements, or as introduced in \cite{Pardue-Richert}, semi-regular sequences.
For a graded Noetherian $k$-algebra  $A$ we say that an element $f\in A_d$ is semi-regular if the multiplication map $A_{i} \xrightarrow{\cdot f} A_{i+d}$ has maximal rank for all $i$. We say that a sequence $f_1,\dots,f_r$ is semi-regular if $f_i$ is semi-regular on $A/(f_1,\dots,f_{i-1})$ for all $1\le i \le r$.

It follows that the Hilbert series of a semi-regular sequence is equal to the series predicted by Fröberg's conjecture.
\begin{conjecture}[Fröberg's conjecture \cite{Frobergs-formodan}]
For a generic ideal $I = (f_1,\dots,f_r) \subset k[x_1,\dots,x_n]$ with $\deg f_i = d_i$ we have
$$
\HS_{R/I}(T) = \left[\frac{\prod_{i=1}^r (1-T^{d_i})}{(1-T)^n}\right].
$$
\end{conjecture}
Here $[.]$ denotes truncation of a power series at the first non-positive coefficient. Moreover the conjecture is equivalent to the statement that a generic ideal is defined by a semi-regular sequence (see \cite{Pardue-Moreno-Socias}). Our reason for choosing to state some results for semi-regular sequences instead of for generic ideals is because we want to make it clear that the statements hold for ideals generated by powers of general linear forms. That $I=(x_1^{d_1},\dots,x_n^{d_n},\ell^{d_{n+1}})$ is a semi-regular sequence follows by first noting that any regular sequence is semi-regular, so $J =(x_1^{d_1},\dots,x_n^{d_n})$ is a semi-regular sequence, and then by Theorem \ref{thm: monom-SLP}, multiplication by $\ell^{d_{n+1}}$ on $R/J$ has maximal rank, which gives the claim.

The following theorem describes the shape of the Hilbert series for an Artinian complete intersection.

\begin{theorem}[\cite{Complete-Intersections}]\label{thm:RRR}
    Let $A = k[x_1,\dots,x_n]/(x_1^{d_1},\dots,x_n^{d_n})$ where $k$ is a field and $1\le d_1\le\dots\le d_n$. Let $t=\sum_{i=1}^n (d_i-1)$.
    
    \begin{itemize}
        \item Suppose $d_n \le \frac{t+1}{2}$. If $t$ is even, then $\HF_A$ is strictly increasing in the interval $[0,t/2]$ and strictly decreasing in the interval $[t/2,t+1]$. If $t$ is odd, then $\HF_A$ is strictly increasing in the interval $[0,\frac{t-1}{2}]$, constant in the interval $[\frac{t-1}{2},\frac{t+1}{2}]$ and strictly decreasing in the interval $[\frac{t+1}{2},t+1]$.
        \item 
        Suppose $d_n > \frac{t+1}{2}$. Then $\HF_A$ is strictly increasing up to a flat top which begins at $t' = \sum_{i=1}^{n-1}(d_i-1)$ and ends at $d_n-1$; afterwards it is strictly decreasing to $0$.
    \end{itemize}

\end{theorem}

A description of the Betti numbers of a semi-regular sequence of $n+1$ forms of degrees $\{d_i\}_{i=1}^{n+1}$, when $\sum_{i=1}^{n+1}(d_i-1) $ is odd, is crucial for our result and appears first in \cite{relatively-compressed}. Their result is however phrased in terms of generic ideals and the description is rather different from our presentation below. The description we give is proved in \cite{Pardue-Richert} for any semi-regular sequence assuming a certain numerical condition, and that this condition is satisfied in our context was shown recently in \cite{froberg2025newbettinumbersideals}. We include a proof sketch following \cite{froberg2025newbettinumbersideals}.

\begin{theorem} \label{thm: Ralf}
Let $I=(f_1,\dots,f_{n+1})$ be a semi-regular sequence in $R=k[x_1,\dots,x_n]$ with $\deg{f_i} = d_i$. If $\sum_{i=1}^{n+1}(d_i-1) $ is odd then the Betti table of $R/I$ coincides, in all but the last row, with that of the Koszul complex $K_\bullet(I)$, and this determines all Betti numbers of $R/I$.
\end{theorem}

\begin{proof}[Proof sketch]
From the results of \cite{Pardue-Richert} and \cite{Diem} it suffices to show that the first non-positive coefficient 
$$
\frac{\prod_{i=1}^{n+1}(1-t^{d_i})}{(1-t)^n}
$$
 is zero, or equivalently that there exists a degree $j$ for which $$
  R/(f_1,\dots,f_n)_{j} \xrightarrow{\cdot f_{n+1}} R/(f_1,\dots,f_n)_{j+d_{n+1}}
 $$ is bijective. By definition, multiplication by $f_{n+1}$ has maximal rank, and if the number of peaks, $s$, of $J$ is greater then $d_{n+1},$ there clearly exists such a degree. Hence we may assume $s <d_{n+1}$. We claim that if $d_{n+1}-s$ is odd, then such a bijection occurs. Since we can write $d_{n+1} = 2k + s-1$, for some integer $k$, it follows that if the peak starts at $i$, then the multiplication by $f_{n+1}$ from degree $i-k$ to $i+(s-1)+k$ is bijective. One checks, by Theorem \ref{thm:RRR}, that this occurs exactly when $\sum_{i=1}^{n+1}(d_i-1)$ is odd.
\end{proof}

 An Artinian $k$-algebra $A=R/I$ is called level if $\text{Soc}(A)$ is generated in one degree, where $\text{Soc}(A) = \{f\in A\mid \mathfrak{m}f = 0\}$. An immediate consequence of Theorem \ref{thm: Ralf}, which will be important to us later, is the following.

\begin{corollary}\label{cor: level}
    Let $I = (f_1,\dots,f_{n+1})$ be a semi regular sequence with $\deg f_i = d_i$. If $\sum_{i=1}^{n+1}(d_i-1)$ is odd. Then $R/I$ is level.
\end{corollary}
\begin{proof}
    By construction we have that $(K_\bullet(I))_n = R(-(\sum_{i=1}^{n+1}d_i))$ and since the socle degree $s$ of $R/I$ is always smaller than $\sum_{i=1}^{n+1}d_i$ it follows from Theorem \ref{thm: Ralf}  that $\beta^R_{n,j}(R/I) = 0$ for $j \le n+s-1$.
\end{proof}

Theorem \ref{thm: Ralf} determines the Betti numbers for $I=(x_1^{d_1},\dots,x_n^{d_n},\ell^{d_{n+1}})$ when $t =\sum_{i=1}^{n+1}(d_i-1)$ is odd, and so the remaining cases are when $\sum_{i=1}^{n+1}(d_i-1)$ is even. Our approach, inspired by \cite{diethorn2025studyquadraticcompleteintersection}, is to show that in the case where at least one $d_i$ equals 2, the Betti numbers  can be reduced to a sum of Betti numbers from a case covered in Theorem \ref{thm: Ralf}, where they are determined uniquely by the Hilbert series.

For $R/G$, where $G = (x_1^{d_1},\dots,x_n^{d_n})\colon (\ell^{d_{n+1}})$, we will determine the Betti numbers in a similar manner, and for this we need a result which determines the Betti numbers of $R/G$ when $\sum_{i=1}^n(d_i-1)$ is odd. This is given in Proposition \ref{prop: G-res-koszul}, but we first need a result from \cite{res-n+1-general-forms} which describes the Hilbert series of $R/G$.

\begin{lemma}[Lemma 2.6 in \cite{res-n+1-general-forms}]\label{lemma: colon ideal lemma}
Let $I=(f_1,\dots,f_{n+1})$ be a minimally generated semi-regular sequence with $\deg f_i = d_i$, let $J$ be the complete intersection $(f_1,\dots,f_n)$ and take $G=J\colon (f_{n+1})$. Then the following holds.
\begin{itemize}
    \item The socle degree of $R/G$ is $\sum_{i=1}^n(d_1-1)-d_{n+1}$.
    \item  For integers $j \le  \frac{\sum_{i=1}^n(d_i-1)-d_{n+1}}{2}$ we have $\HF_{R/G}(j) = \HF_{R/J}(j)$. By symmetry, this completely determines $\HS_{R/G}(T)$.
\end{itemize}
\end{lemma}

The following proposition can also be stated in terms of relatively compressed algebras, since we only use the fact that $R/G$ will be relatively compressed with respect to the complete intersection $J$. It then corresponds to Theorem 3.5 from \cite{relatively-compressed}. We choose to include a new proof since it becomes rather simple using Lemma \ref{lemma: pr-truncated-resolution}.

\begin{prop}\label{prop: G-res-koszul}
Let $(f_1,\dots,f_{n+1})$, with $\deg f_i = d_i$ be a minimally generated semi-regular sequence, set $s=\frac{(\sum_{i=1}^nd_i) - d_n -n}{2}$, $J= (f_1\dots f_n) $ and $G= J\colon (f_{n+1})$. If $\sum_{i=1}^{n+1}(d_i -1)$ is odd, then the Betti table of $R/G$ is equal to the Betti table from the Koszul complex $K_\bullet(J)$ in the first $s$ rows, and this determines all Betti numbers of $R/G$.
\end{prop}

\begin{proof}    
    By Lemma \ref{lemma: colon ideal lemma} we have that $\HF_{R/G}(d) = \HF_{R/J}(d)$ for $d \le s$, and since $J \subseteq G$ we must have $J_{(s)} = G_{(s)}$. Note that since $J$ is a complete intersection $R/J$ is resolved by the Koszul complex $K_\bullet(J)$. If we let $F_\bullet$ be the minimal free resolution of $R/G$ then $G$ is resolved by $\overline{F}_\bullet$, and $J$ by $\overline{K}_\bullet(J)$, where the line denotes that we exclude the last free module in the resolution. It follows by Lemma \ref{lemma: pr-truncated-resolution} that $\overline{F}_\bullet^{(s)} \cong (\overline{K}_\bullet(J))^{(s)}$ and that the diagram 
    $$
    \begin{tikzcd}
         (\overline{F}_\bullet)^s \arrow[r] & G_{(s)} \arrow[r] & R\\
         (\overline{K}_\bullet(J))^{(s)} \arrow[r] \arrow[u,"\sim" {rotate=-90, anchor=north}] & J_{(s)} \arrow[u,"id"] \arrow[r] & R \arrow[u,"1_R"]
    \end{tikzcd} 
    $$ commutes. Hence $F_\bullet^{(s-1)} \cong (K_\bullet(J))^{(s-1)}$. It follows that the Betti table of $R/I$ is equal to the Betti table of $K_\bullet(J)$ in the first $s$ rows. Since $R/G$ is Gorenstein, with socle degree $2s$, it follows by symmetry that this determines all but the middle row of the Betti table, but these values are then forced by the Hilbert series of $R/G$.
    \end{proof}

\begin{remark}
    It is worth pointing out that there are weaker versions of both Theorem \ref{thm: Ralf} and Proposition \ref{prop: G-res-koszul} which hold without the assumption on the degrees. It says that the Betti table of $R/I$ is equal to that of the Koszul complex in all but the two last rows, and similarly that the Betti numbers of $R/G$ are equal to that of $K_\bullet(J)$ in all but the two middle rows. The first statement can be found in \cite{ubiquity}, \cite{Pardue-Richert} and a generalization in \cite{Diem}. Regarding the second statement one finds the corresponding result regarding $J$-compressed algebras in Theorem 3.9 of \cite{relatively-compressed}. The proof we give of Proposition \ref{prop: G-res-koszul} is easily adapted to give an alternative proof of the latter.
\end{remark}

Essential to our approach is also the theory of liftable modules (see \cite{Lifting-module}). For $s\in R$ an $R/(s)$-module $M$ is said to be liftable to $R$ if there exists an $R$-module $L$ such that $R/(s) \otimes_R L \cong M$ and $\Tor^R_i (L,R/(s)) =0$ for all $i>0$. A module is called weakly liftable (to $R$) if it is the direct summand of some liftable module. In our case we will consider $s$ to be a non-zero divisor on $R$, and the condition $\Tor^R_i (L,R/(s)) =0$, for $i>0$, is then equivalent to $s$ being regular on $L$.

For a non-zero divisor $s \in R$ the natural surjection $R \to R/(s)$ gives a long exact sequence in Tor (\cite{Rotman-Homological} 11.71)
\begin{equation}\label{eq: long-exact-tor}
\begin{tikzcd}[sep=small]
\dots \arrow[r] &\Tor_{i-1}^{R/(s)}(M,N) \arrow[r]
& \Tor_{i}^R(M,N) \arrow[r]
\arrow[d, phantom, ""{coordinate, name=Z}]
& \Tor_{i}^{R/(s)}(M,N) \arrow[dll,
"\delta_i",
rounded corners,
to path={ -- ([xshift=2ex]\tikztostart.east)
|- (Z) [near end]\tikztonodes
-| ([xshift=-2ex]\tikztotarget.west)
-- (\tikztotarget)}] \\ &
\Tor_{i-2}^{R/(s)}(M,N) \arrow[r]
& \Tor_{i-1}^{R}(M,N) \arrow[r]
& \Tor_{i-1}^{R/(s)}(M,N) \arrow[r] & \dots,
\end{tikzcd}
\end{equation} for all $R/(s)$-modules $M,N$.

The reason for taking interest in liftability in our case is because of the following result.

\begin{prop}[Proposition 3.1 in \cite{Lifting-module}]\label{prop: liftable}
For an $R/(s)$ module $M$, the connecting map $\delta_i$ is zero for all $R/(s)$-modules $N$ and $i>0$,  if and only if $M$ is weakly liftable to $R$.
\end{prop}

\section{Resolutions of $R/I$ and $R/G$}
For the rest of the paper we use the notation 
\begin{align*}
R=k[x_1,\dots,x_n], \quad  \overline{R}=k[x_1,\dots,x_{n-1}], \\\ell = x_1+\dots+x_n \quad \text{and}  \quad \overline{\ell}=x_1+\dots+x_{n-1},
\end{align*}
 where $k$ is a field of characteristic 0.
In this section we consider ideals generated by $n+1$ powers of general linear forms where at least one generators is a square. For notational purposes we assume that the power of $x_n$ is a square, and so we have $I = (x_1^{d_1},\dots,x_{n-1}^{d_{n-1}},x_n^2,\ell^{d_n})$ where $d_i$ for $1\le i\le n$ are allowed to be any positive integers. We will always, even if not specified, assume that $I$ is minimally generated by $(x_1^{d_1},\dots,x_{n-1}^{d_{n-1}},x_n^2,\ell^{d_n})$. If not minimally generated then $I$ would be a complete intersection and hence the minimal free resolution is already well known. Similarly we let $G = (x_1^{d_1},\dots,x_{n-1}^{d_{n-1}},x_n^2) \colon (\ell^{d_n})$, which is the Artinian Gorenstein algebra linked to $I$ by the complete intersection $J = (x_1^{d_1},\dots,x_{n-1}^{d_{n-1}},x_n^2)$.
The plan is to show that for the cases when $\sum_{i=1}^{n}(d_i-1)$ is odd we can, using liftability from $R/(x_n^2)$ to $R$, reduce the Betti numbers of $R/I$ and $R/G$ to a case where the Betti numbers are known. The following provides a sketch of the main idea.

Let $M$ be an $R/(x_n^2)$ module (for us it will be either $R/I$ or $R/G$) which is liftable to $R$. Then, by Proposition \ref{prop: liftable} we split the long exact sequence \eqref{eq: long-exact-tor} into short exact sequences of the form
$$
    0 \to  \Tor_{i-1}^{R/(x_n^2)}(M,k)_{j-2} \to \Tor_{i}^{R}(M,k)_j \to
 \Tor_{i}^{R/(x_n^2)}(M,k)_j \to 0.
$$

Since $\beta_{i,j}^R(M) = \dim_k \Tor^R_i(M,k)_j$, and similarly over $R/(x_n^2)$ we obtain that
$$
\beta_{i,j}^R(M) = \beta_{i,j}^{R/(x_n^2)}(M) + \beta_{i-1,j-2}^{R/(x_n^2)}(M).
$$

Lastly we show that $\beta_{i,j}^{R/(x_n^2)}(M)$ correspond to the Betti numbers over $R$ of another ideal, which we understand from Theorem \ref{thm: Ralf} and Proposition \ref{prop: G-res-koszul}.

Thus in order to prove our main result, Theorem \ref{thm: sum of betti}, we show that $R/I$ and $R/G$ are liftable to $R$ when  $t = \sum_{i=1}^n(d_i-1)$ is odd, where $I=(x_1^{d_1},\dots,x_{n-1}^{d_{n-1}},x_n^2,\ell^{d_n})$, and $G =(x_1^{d_1},\dots,x_{n-1}^{d_{n-1}},x_n^2)\colon (\ell^{d_n})$. This is primarily done in Lemma $\ref{lemma:regular}$.

For this we will need some propositions and lemmas to help us.
In what follows we use the notation $[T^a]f$ of some polynomial $f \in k[T]$ to denote the coefficient of $T^a$ in $f$.

\begin{lemma}\label{lem:multi}
Let $I$ be minimally generated by $ (x_1^{d_1},\dots,x_{n-1}^{d_{n-1}},x_n^2,\ell^{d_n})$, and set $t = \sum_{i=1}^{n}(d_i-1)$. If $t$ is odd, then we have
$$
e(R/(I+(x_n))) = \dim_kR/(I+(x_n)) = \left[T^{\frac{t-1}{2}}\right]\frac{\prod_{i=1}^{n}(1-T^{d_i})}{(1-T)^n}.
$$
\end{lemma}
\begin{proof}
    The first equality follows by definition of multiplicity since $R/I$ is Artinian. Without of loss of generality we can assume that $d_n = \max_{i}d_i$. Since $I$ is minimally generated we have that $d_n \le \sum_{i=1}^{n-1}(d_i-1) + 1$, and if this is an equality we have that $t$ is even and so we may assume that $d_n \le \sum_{i=1}^{n-1}(d_i-1) $. Theorem \ref{thm:RRR} then tells us that the Hilbert function of the complete intersection $J = (x_1^{d_1},\dots,x_{n-1}^{d_{n-1}},\ell^{d_n})$ has exactly two peaks, which occur at degree $\frac{t-1}{2}$ and $\frac{t+1}{2}$. By $J$ having WLP (in fact SLP) with $x_n$ as a weak Lefschetz element it then follows that 
    $$
    \HF_{R/(I+(x_n))}(i) = \HF_{R/J}(i) -\HF_{R/J}(i-1),
    $$ for $i\le \frac{t-1}{2}$ and $\HF_{R/(I+(x_n))}(i)=0$ for $i>\frac{t-1}{2}$. Hence 
    \begin{align*}  
    \dim_k(R/(I+(x_n))) = \sum_{i=0}^{\frac{t-1}{2}}\HF_{R/(I+(x_n))}(i) = \\   \HF_{R/J}\left(\frac{t-1}{2}\right) = \left[T^{\frac{t-1}{2}}\right]\frac{\prod_{i=1}^{n}(1-T^{d_i})}{(1-T)^n}.
    \end{align*}
    \end{proof}

\begin{lemma}\label{proj_set}
    Consider $d_1,\dots,d_n \in \mathbb{N}$ with $t=\sum_{i=1}^{n}(d_i-1)$ being odd. Let $X \subset \mathbb{P}^{n-1}(k)$ be the set represented by elements $(a_1,\dots,a_{n-1},1)$ where $$a_i \in \{(d_i-1)-2j \mid 0\le j\le d_i-1  \}$$ and such that $$1+\sum_{j=0}^{n-1}a_i \in \{(d_n-1)-2j \mid 0 \le j\le  d_n-1  \}.$$ Then we have 
    $$
    |X| =  \left[T^{\frac{t-1}{2}}\right]\frac{\prod_{i=1}^{n}(1-T^{d_i})}{(1-T)^n}.
    $$
\end{lemma}
\begin{proof}
    Let $m_{i} = T^{-(d_i-1)} + T^{-(d_i-3)} + \dots + T^{d_i-1} $, so that the exponents corresponds to the possible values of $a_i$. Then the generating function whose coefficients give the number of vectors $a$ summing to the corresponding exponent is given by
    $$
    g=T\prod_{i=1}^{n-1}m_{i}.
    $$

    After setting $Y=T^2$, we may rewrite each factor as 
    $$
    m_{i} = T^{-(d_i-1)}(1+Y+\dots + Y^{d_i-1}) = T^{-(d_i-1)}\frac{(1-Y^{d_i})}{(1-Y)}, 
    $$ and hence
    $$
    g = T^{-(\sum_{i=1}^{n-1}(d_i-1)-1)}\frac{\prod_{i=1}^{n-1}(1-Y^{d_i})}{(1-Y)^{n-1}}.
    $$
    We are however interested in the sum of the coefficients from the terms with exponents in $\{d_n-1 -2j\mid 0\le j \le d_n-1   \}$, which can be collected in the coefficient of $T^{d_n-1}$ by multiplying with $(1+Y+\dots +Y^{d_n-1})$. This gives
    $$
    g(1+Y+\dots +Y^{d_n-1}) = T^{-(\sum_{i=1}^{n-1}(d_i-1)-1)}\frac{\prod_{i=1}^{n}(1-Y^{d_i})}{(1-Y)^n}.
    $$

    Lastly we note that the coefficient of $T^{d_n-1}$ of the polynomial above is the same as the coefficient of $T^{\sum_{i=1}^n(d_i-1) -1} = T^{t-1}$ in $\frac{\prod_{i=1}^{n}(1-Y^{d_i})}{(1-Y)^n}$, or equivalently the coefficient of $Y^{\frac{t-1}{2}}$ in $$\frac{\prod_{i=1}^{n}(1-Y^{d_i})}{(1-Y)^n}.$$ Since each counted $a$ is distinct in $\mathbb{P}^{n-1}(k)$ (as the last coordinate is fixed) we conclude that 
    $$
    |X| = \left[Y^{\frac
    {t-1}{2}}\right]\frac{\prod_{i=1}^{n}(1-Y^{d_i})}{(1-Y)^n}.
    $$ \end{proof}

With this we are ready, as outlined at the start of this section, for the main construction. In \cite{diethorn2025studyquadraticcompleteintersection} the authors consider the ideal $P=(x_1^2-x_n^2,\dots,x_{n-1}^2-x_n^2,\ell^2-x_n^2)$, when $n$ is odd, and show that $x_n$ is regular on $R/P$. Together with $P+(x_n^2) =(x_1^2,\dots,x_{n}^2,\ell^2) $, this shows that $R/(P+(x_n^2))$ is liftable to $R$. The main property of $P$ is that its vanishing locus has cardinality equal to the multiplicity of $R/(P+(x_n))$. We wish to find some ideal $I'$ with these same properties in this more general setup. It turns out that the following construction is the ``correct" generalization.

For $1\le i < n$ we define 
$$
f_{i}= \prod_{j=0}^{d_i-1}(x_i-((d_i-1)-2j)x_n),
$$ and we also set 

$$
f_{n} =\prod_{j=0}^{d_n-1}(\ell-((d_i-1)-2j)x_n).
$$

\begin{remark}\label{remark: L-constr}
     If $d_i$ is even then we have $f_{i} =\prod_{j=1}^{d_i/2}(x_i^2-(2j-1)^2x_n^2) $, when $d_i$ is odd we have $f_{i} =x_i\prod_{j=1}^{(d_i-1)/2}(x_i^2-(2j)^2x_n^2)$, and similarly for $f_{n}$. When setting all $d_i=2$ we obtain the construction from \cite{diethorn2025studyquadraticcompleteintersection}, mentioned above.
\end{remark}

We define the ideals
$$J' = (f_{1},\dots,f_{n-1})\text{, } \ I' = (f_{1},\dots,f_{n} ) \quad \text{and} \quad G' = (f_1,\dots,f_{n-1})\colon (f_n).
$$ From Remark \ref{remark: L-constr} it is easy to see that $R/(x_n^2) \otimes_R R/I' \simeq R/(I'+(x_n^2)) = R/I$. Note that we have the short exact sequence
$$
0 \to R/G'[-d_n] \xrightarrow{\cdot f_n} R/J' \to R/I' \to 0.
$$

The following lemmas, the first of which generalize Lemma 4.5 in \cite{diethorn2025studyquadraticcompleteintersection}, justify the construction. 

\begin{lemma}\label{lemma:regular}
Let $t= \sum_{i=1}^n(d_i-1)$ be odd. Then $x_n$ is a regular element on $L=R/I'$ and on $R/G'$.
\end{lemma}
\begin{proof}
We naturally have the exact sequence
$$
0 \xrightarrow{} K  \xrightarrow{} R/I' \ [-1] \xrightarrow{\cdot x_n} R/I' \xrightarrow{} R/(I'+(x_n)) \xrightarrow{} 0
$$ and are left to prove that $K = 0$. Note that $I'+(x_n) = I+(x_n)$.

Let $X\subset \mathbb{P}^{n-1}(k)$ denote the vanishing locus of $I'$. Since each factor of the defining polynomials are linear one finds that that $X$ is exactly the projective set defined in Lemma \ref{proj_set}.

We let $I_X$ be the corresponding vanishing ideal of X. Then $I' \subseteq I_X $, and moreover $\dim R/I' = \dim R/I_X = 1$. Thus we can write $\HS_{R/I'}(T) = \frac{q(T)}{1-T}$ with $e(R/I') = q(1)$. Since $K$ and $R/(I+(x_n))$ are Artinian we have
$$
e(R/I') = \HS_{R/(I+(x_n))}(1) - \HS_K(1) =  e(R/(I+(x_n))) - e(K)
$$ and so $e(R/(I+(x_n))) \ge e(R/I')$ (with equality if and only if $K = 0$). From $I'\subseteq I_X$ we also get $e(R/I_X) \le e(R/I')$.

Now it is a standard fact that since $X$ is a finite projective set we have $e(R/I_X)=|X|$. Lemma \ref{proj_set} says that  $$|X|= \left[T^{\frac
    {t-1}{2}}\right]\frac{\prod_{i=1}^{n}(1-T^{d_i})}{(1-T)^n},
    $$ and using Lemma \ref{lem:multi} we find that $e(R/(I+x_n))=|X|$. Hence $K=0$, and we conclude that $x_n$ is regular on $L$.

For the second statement we find that the vanishing locus of  the ideal $J' =  (f_{1},\dots,f_{{n-1}})$ is the projective set $X'$ of vectors which can be represented on form $(a_1,\dots,a_{n-1},1)$, with $a_i \in \{(d_i-1)-2j) \mid 0\le j\le (d_i-1)  \}$, which has cardinality $|X'| = \prod_{i=1}^{n-1}d_i$. The Artinian complete intersection $J'+(x_n) = (x_1^{d_1},\dots,x_{n-1}^{d_{n-1}},x_n)$ has Hilbert series 
$$
\HS_{R/(J+(x_n))}(T) = \frac{ \prod_{i=1}^{n-1}(1-T^{d_i})}{(1-T)^{n-1}}
$$ and hence also multiplicity $\prod_{i=1}^{n-1}d_i$. Consider the exact sequence 
$$
0 \xrightarrow{} K'  \xrightarrow{} R/J' \ [-1] \xrightarrow{\cdot x_n} R/J' \xrightarrow{} R/(J'+(x_n)) \xrightarrow{} 0.
$$ Then $\dim R/J' =1$ so we have $e(R/J') = \HS_{R/(J'+(x_n))}(1) - \HS_{K'}(1) = e(R/(J'+(x_n))) - e(K')$. Again $J'\subset I_{X'}$ so $e(R/(J')) \ge e(R/I_X) = |X'| = e(R/(J'+(x_n)))$, and we must have $K'=0$.

 This gives that $x_n$ is regular on $R/J'$ and thus also on $R/[J':(f_{n})] = R/G'$, as it injects into $R/J'$.
\end{proof}

\begin{lemma}\label{lemma: canonical_isosny} Let $I$ be minimally generated by $(x_1^{d_1},\dots,x_{n-1}^{d_{n-1}},x_n^2,\ell^{d_n})$, let $G= (x_1^{d_1},\dots,x_{n-1}^{d_{n-1}},x_n^2)\colon (\ell^{d_n})$ and $\overline{G}  = (x_1^{d_1},\dots,x_{n-1}^{d_{n-1}})\colon (\overline{\ell}^{d_n})$. When $t= \sum_{i=1}^n (d_i-1)$ is odd, we have that
    \begin{itemize}
        \item $ G = G'+(x_n^2)$,
         \item  $ \ov{R}{G} \cong R/(G'+(x_n)) $ as graded $\overline{R}$-modules.
    \end{itemize}

\end{lemma}
\begin{proof}
Using that $x_n$ is regular on $L = R/I'$, by Lemma \ref{lemma:regular}, the proof is analogous to Lemma 4.6 in \cite{diethorn2025studyquadraticcompleteintersection}: Starting with the short exact sequence
$$
0 \to R/G'[-d_n] \xrightarrow{\cdot f_n} R/J' \to R/I' \to 0
$$ and tensoring with $R/(x_n^2)$ we obtain, after taking homology,  the exact sequence
$$
\Tor_1^R(R/I', R/(x_n^2)) \to R/(G'+(x_n^2))[-d_n] \xrightarrow{\cdot f_n} R/J \to R/I \to 0.
$$ Now in $R/(x_n^2)$ we have $f_n = \ell^{d_n}$ and moreover since $x_n$ is regular on $L=R/I'$ we have $\Tor_1^R(R/I', R/(x_n^2)) = 0$. Hence we have the short exact sequence
$$
0\to R/(G'+(x_n^2))[-d_n] \xrightarrow{\cdot \ell^{d_n}} R/J \to R/I \to 0
$$ and conclude that $G'+(x_n^2) = J\colon(\ell^{d_n}) = G$.

For the second statement we also begin with 
$$
0 \to R/G'[-d_n] \xrightarrow{\cdot f_n} R/J' \to R/I' \to 0.
$$ Tensoring with $R/(x_n)$ gives 
$$
0 \to R/(G'+(x_n))[-d_n] \xrightarrow{\cdot f_n} R/(J+(x_n)) \to R/(I+(x_n)) \to 0 
$$ since $\Tor_1^R(R/I',R/(x_n)) = 0$, as $x_n$ is regular on $R/I'$. Using that $R/(J+(x_n)) \cong \ov{R}{J}$ and $R/(I+(x_n)) \cong \ov{R}{I}$, where $\overline{J}=(x_1^{d_1},\dots,x_{n-1}^{d_{n-1}})$ and $\overline{I}=\overline{J}+\ell^{d_n}$, as $\overline{R}$-modules we find that
$$
0 \to R/(G'+(x_n))[-d_n] \xrightarrow{\cdot \ell^{d_n}} \ov{R}{J} \to \ov{R}{I} \to 0 
$$ is exact. Hence $R/(G'+(x_n)) \cong \overline{R}/({\overline{J}}\colon (\overline{\ell}^{d_n}))=\ov{R}{G}$.
\end{proof}

\begin{prop}\label{prop: colon ideals}
Let $I =(x_1^{d_1},\dots,x_{n-1}^{d_{n-1}},x_n^2,\ell^{d_n})$ be minimally generated, with $t=\sum_{i=1}^n(d_i-1)$ being odd. Then
$$
\begin{tikzcd}
    0 \arrow[r] & R/(I+(x_n))[-1] \arrow[r,"\cdot x_n"] & R/I \arrow[r] & R/(I+x_n) \arrow[r] & 0 \\
     0 \arrow[r] & R/(G+(x_n))[-1] \arrow[r,"\cdot x_n"] & R/G \arrow[r] & R/(G+x_n) \arrow[r] & 0
\end{tikzcd}
 $$
 are exact, or equivalently $I \colon (x_n) = I +(x_n)$ and $G:(x_n)  = G+ (x_n)$.
\end{prop}
\begin{proof}
This follows directly from that $x_n$ is regular on $R/I'$ and $R/G'$ (Lemma \ref{lemma:regular}) together with $I' + x_n^2 = I$ (Remark \ref{remark: L-constr}) and $G'+x_n^2 = G$ (Lemma \ref{lemma: canonical_isosny}). Since $I+(x_n) \subset I:(x_n)$ and similarly for $G$ we only need to show the other inclusion.
If $g \in I:(x_n) $ then $gx_n \in I = I' + (x_n^2)$ and hence $g \in I' + (x_n) = I+(x_n)$. An identical argument gives the statement for $G$.
\end{proof}

The last results which is needed for the main result is the following.

\begin{prop}\label{prop: betti number change}
Let $I = (x_1^{d_1},\dots,x_{n-1}^{d_{n-1}},x_n^2,\ell^{d_n})$ be minimally generated, $G = (x_1^{d_1},\dots,x_{n-1}^{d_{n-1}},x_n^2)\colon (\ell^{d_n})$, and assume that $t= \sum_{i=1}^n(d_i-1)$ is odd. Then
\begin{align*}
    \beta_{i,j}^{R/(x_n^2)}(R/I) &=\beta_{i,j}^{\overline{R}}(\ov{R}{I}) \quad \text{and}\\
    \beta_{i,j}^{R/(x_n^2)}(R/G) &=\beta_{i,j}^{\overline{R}}(\ov{R}{G}). 
\end{align*}
\end{prop}
\begin{proof}
    Using that $I\colon(x_n) = I+(x_n)$ and similarly for $G$ (Proposition \ref{prop: colon ideals}), the proof is analogous to Proposition 4.8 in \cite{diethorn2025studyquadraticcompleteintersection}:

    Consider the following minimal graded free resolution of $R/(x_n)$ over $R/(x_n^2)$
$$
\begin{tikzcd}
    \dots \arrow[r] & R/(x_n^2)[-2] \arrow[r,"\cdot x_n"] & R/(x_n^2)[-1]\arrow[r,"\cdot x_n"] & R/(x_n^2) \arrow[r] & 0.
\end{tikzcd}
$$ Tensoring with $R/I$ gives 
$$
\begin{tikzcd}
    \dots \arrow[r] & R/I[-2] \arrow[r,"\cdot x_n"] & R/I[-1]\arrow[r,"\cdot x_n"] & R/I \arrow[r] & 0,
\end{tikzcd}
$$ which is also exact (in all but the first map) since  Proposition \ref{prop: colon ideals} says that $I:(x_n) = I+(x_n)$. Thus $\Tor_i^{R/(x_n^2)}(R/(x_n) , R/I) = 0$ for $i>0$,
and if we let $G^\bullet $ denote a minimal free graded resolution of $R/I$ over $R/(x_n^2)$ then, for $i>0$  $$H^i(R/(x_n) \otimes_{R/(x_n^2)} G^\bullet) = \Tor_i^{R/(x_n^2)}(R/(x_n),R/I) = 0.$$ But $R/(x_n) \otimes_{R/(x_n^2)} G^\bullet $ is then also a (minimal) free graded resolution of $$R/(x_n) \otimes_{R/(x_n^2)} R/I \cong R/(I+(x_n)) \cong \ov{R}{I}$$ over $R/(x_n) \cong \overline{R}$ and hence the statement follows. The proof for $R/G$ is identical using that we have $R/(G+(x_n))\cong \ov{R}{G}$ (Lemma \ref{lemma: canonical_isosny}).
\end{proof}

We are now ready to present the main result.

\begin{theorem}\label{thm: sum of betti}
    Let $I$ be minimally generated by $(x_1^{d_1},\dots,x_{n-1}^{d_{n-1}},x_n^2,\ell^{d_{n}})$, let $G = (x_1^{d_1},\dots,x_{n-1}^{d_{n-1}},x_n^2)\colon(\ell^{d_n})$, and assume that $t= \sum_{i=1}^n(d_i-1)$ is odd. Then
    \begin{align*}
        \beta_{i,j}^R(R/I) &=  \beta_{i,j}^{\overline{R}}(\ov{R}{I}) +  \beta_{i-1,j-2}^{\overline{R}}(\ov{R}{I}) \\
     \beta_{i,j}^R(R/G) &=  \beta_{i,j}^{\overline{R}}(\ov{R}{G}) +  \beta_{i-1,j-2}^{\overline{R}}(\ov{R}{G}),      
    \end{align*}
    where $\overline{I} = (x_1^{d_1},\dots,x_{n-1}^{d_{n-1}},\overline{l}^{d_n})$ and $\overline{G} = (x_1^{d_1},\dots,x_{n-1}^{d_{n-1}})\colon (\overline{l}^{d_n})$.
    Moreover, the Betti numbers of $\ov{R}{I}$ and $\ov{R}{G}$ are determined by Theorem \ref{thm: Ralf} and Proposition \ref{prop: G-res-koszul}.
\end{theorem}
\begin{proof}
Remark \ref{remark: L-constr} gives that $R/(I'+(x_n^2)) = R/I$ and Lemma \ref{lemma: canonical_isosny} that $R/(G'+(x_n^2)) = R/G$. Together with Lemma \ref{lemma:regular}, which says that $x_n$ is regular on both $R/I'$ and $R/G'$, we have that both $R/G$ and $R/I$ are liftable to $R$. Hence, by Proposition \ref{prop: liftable}, we obtain short exact sequences
 $$
 0 \xrightarrow{} \Tor_{i-1}^{R/(x_n^2)}(R/I,k)[-2] \xrightarrow{} \Tor_i^R(R/I,k) \xrightarrow{} \Tor_i^{R/(x_n^2)}(R/I,k) \to 0
 $$ for all $i$, and similarly for $R/G$.

 Restricting to homogeneous degree $j$, and using Proposition \ref{prop: betti number change},  we obtain 
 $$
 \beta_{i,j}^R(R/I) = \beta_{i,j}^{{R/(x_n^2)}}({R}/{I}) +  \beta_{i-1,j-2}^{{R/(x_n^2)}}({R}/{I}) = \beta_{i,j}^{\overline{R}}(\ov{R}{I}) +  \beta_{i-1,j-2}^{\overline{R}}(\ov{R}{I}).
 $$

 and similarly for $R/G$. Lastly, since $\sum_{i=1}^{n}(d_i-1)$ is odd, the Betti numbers of $\ov{R}{I}$ and $\ov{R}{G}$ follow directly from Theorem \ref{thm: Ralf} and Proposition \ref{prop: G-res-koszul}.
\end{proof}

Together with Theorem \ref{thm: Ralf} and Proposition \ref{prop: G-res-koszul} this determines the Betti numbers for any ideal of the form  $I = (x_1^{d_1},\dots,x_{n-1}^{d_{n-1}},x_n^2,\ell^{d_{n}})$, and $G =(x_1^{d_1},\dots,x_{n-1}^{d_{n-1}},x_n^2)\colon (\ell^{d_{n}})$. An easy but interesting consequence of Theorem \ref{thm: sum of betti} is the following.

\begin{corollary}\label{cor: all is level}
    Let $I= (x_1^{d_1},\dots,x_{n-1}^{d_{n-1}},x_n^2,\ell^{d_{n}})$. Then the algebra $R/I$ is level.
\end{corollary}
\begin{proof}
    If $t= \sum_{i=1}^{n}(d_i-1)$ is even then this follows by Corollary $\ref{cor: level}$. If $t$ is odd we find by Theorem $\ref{thm: sum of betti}$ that 
    $$
    \beta^R_{n,j}(R/I) = \beta_{n-1,j-2}^{\overline{R}}(\ov{R}{I}).
    $$ Since $\overline{I} =(x_1^{d_1},\dots,x_{n-1}^{d_{n-1}},\overline{\ell}^{d_{n}}) $ is level, by Corollary $\ref{cor: level}$, it follows that so is $R/I$.
\end{proof}

\begin{example}
The Betti numbers of the ideal $\overline{I} =(x_1^4,x_2^4,x_3^4,x_4^4,\overline{\ell}^4)$ are determined by its Hilbert function, by Theorem \ref{thm: Ralf}, and is given on the left below. From Theorem \ref{thm: sum of betti} we then obtain the Betti numbers of $I =(x_1^4,x_2^4,x_3^4,x_4^4,x_5^2,\ell^4)$, which are displayed to the right.

 \begin{align*}\begin{matrix}
        & 0 & 1 & 2 & 3 & 4\\
       \text{total:} & 1 & 5 & 30 & 46 & 20\\
       0: & 1 & . & . & . & .\\
       1: & . & . & . & . & .\\
       2: & . & . & . & . & .\\
       3: & . & 5 & . & . & .\\
       4: & . & . & . & . & .\\
       5: & . & . & . & . & .\\
       6: & . & . & 10 & . & .\\
       7: & . & . & 20 & 46 & 20
       \end{matrix} \quad \quad 
       \begin{matrix}
        & 0 & 1 & 2 & 3 & 4 & 5\\
       \text{total:} & 1 & 6 & 35 & 76 & 66 & 20\\
       0: & 1 & . & . & . & . & .\\
       1: & . & 1 & . & . & . & .\\
       2: & . & . & . & . & . & .\\
       3: & . & 5 & . & . & . & .\\
       4: & . & . & 5 & . & . & .\\
       5: & . & . & . & . & . & .\\
       6: & . & . & 10 & . & . & .\\
       7: & . & . & 20 & 56 & 20 & .\\
       8: & . & . & . & 20 & 46 & 20
       \end{matrix}.\end{align*}
Similarly, the Betti tables of the corresponding colon ideals $\overline{G}$ and $G$ are given, respectively, by

\begin{align*}
\begin{matrix}
        & 0 & 1 & 2 & 3 & 4\\
       \text{total:} & 1 & 24 & 46 & 24 & 1\\
       0: & 1 & . & . & . & .\\
       1: & . & . & . & . & .\\
       2: & . & . & . & . & .\\
       3: & . & 4 & . & . & .\\
       4: & . & 20 & 46 & 20 & .\\
       5: & . & . & . & 4 & .\\
       6: & . & . & . & . & .\\
       7: & . & . & . & . & .\\
       8: & . & . & . & . & 1
       \end{matrix} \quad \quad
    \begin{matrix}
        & 0 & 1 & 2 & 3 & 4 & 5\\
       \text{total:} & 1 & 25 & 70 & 70 & 25 & 1\\
       0: & 1 & . & . & . & . & .\\
       1: & . & 1 & . & . & . & .\\
       2: & . & . & . & . & . & .\\
       3: & . & 4 & . & . & . & .\\
       4: & . & 20 & 50 & 20 & . & .\\
       5: & . & . & 20 & 50 & 20 & .\\
       6: & . & . & . & . & 4 & .\\
       7: & . & . & . & . & . & .\\
       8: & . & . & . & . & 1 & .\\
       9: & . & . & . & . & . & 1
       \end{matrix}
\end{align*}

\end{example}

A fascinating consequence of Lemma \ref{lemma:regular} is that the syzygies of $I$ are more structured then one might guess. For a complete intersection it is known that (since the Koszul complex is the minimal free resolution) the components of the relations themselves lie in the ideal. However for most ideals, the syzygies tend to be very complicated. It turns out that the ideals we are studying satisfy a property similar to that of the Koszul relations.

\begin{corollary}\label{cor: syzygies}
Let $I = (x_1^{d_1},\dots,x_{n-1}^{d_{n-1}},x_n^2,\ell^{d_n})$ and let $(a_1,\dots,a_{n+1})\in R^{n+1}$ be a relation of $I$, such that $a_1x_1^{d_1} + \dots a_{n-1}x_{n-1}^{d_{n-1}} +a_nx_n^2 + a_{n+1}\ell^{d_n} = 0$. Then, if $\sum_{i=1}^n(d_i-1)$ is odd, we have
$$
c_1a_1x_1^{d_1-2} + \dots +c_{n-1}a_{n-1}x_{n-1}^{d_{n-1}-2} + a_n + c_{n}a_{n+1}\ell^{d_{n}-2} \in I,
$$ where $c_i = \sum_{i=0}^{\lfloor\frac{d_i-1}{2}\rfloor} (d_i-1-2j)^2$.

Most notably, for the fully quadratic case, when $n$ is odd any relation $(a_1,\dots,a_{n+1})$ of the ideal $I=(x_1^2,\dots,x_n^2,\ell^2)$ satisfies 
$$
a_1+\dots+a_{n+1} \in I.
$$

\end{corollary}
\begin{proof}
    This result is tightly related to the fact that $x_n$ is regular on $L=R/I'$.
As per the definition of $L$ we set $f_{i}= \prod_{j=0}^{d_i-1}(x_i-((d_i-1)-2j)x_n)$ and similarly $f_{n}= \prod_{j=0}^{d_i-1}(\ell-((d_i-1)-2j)x_n)$. By Remark \ref{remark: L-constr} we have $f_{i} = x_i^{d_i} - c_ix_i^{d_i-2}x_n^2 + r_ix_n^4$, for some $r_i \in R_{d_i-4}$. For any relation $$
a_1x_1^{d_1} + \dots +a_nx_n^2 + a_{n+1}\ell^{d_n} = 0,
$$ we can rewrite it as
\begin{align*}
    &a_1f_{1} + \dots + + a_{n-1}f_{n-1} + a_{n+1}f_{n}   \\ &+x_n^2\left(a_n -\sum_{i=1}^{n-1}a_i\frac{(f_{i}-x_i^{d_i})}{x_n^2} - a_{n+1}\frac{(f_{n}- \ell^{d_i})}{x_n^2}\right)   = 0.
\end{align*}

Lemma \ref{lemma:regular} tells us that $x_n$ is regular on $L=R/(f_1,\dots,f_n)$, and hence 
$$
a_n -\sum_{i=1}^{n-1}a_i\frac{(f_{i}-x_i^{d_i})}{x_n^2} - a_{n+1}\frac{(f_{n}- \ell^{d_i})}{x_n^2} \in (f_{1},\dots,f_{n}).
$$ Since $\frac{(f_{i}-x_i^{d_i})}{x_n^2} = -c_ix_i^{d_i-2} + r_ix_n^2$, we obtain that
$$
\sum_{i=1}^{n-1}a_ic_ix_i^{d_i-2} + a_n + c_{n}a_{n+1}\ell^{d_{n}-2} \in (f_{1},\dots,f_{n}) + (x_n^2) = I.
$$
\end{proof}

Lastly, using Corollary \ref{cor: all is level}, we also give an application to generic ideals. A special case, when all generators are quadrics, of the following result was very recently proved in \cite{diethorn2026homologicalpropertiesringsdefined}.

\begin{corollary}\label{cor: generic level}
A generic ideal of $k[x_1,\dots,x_n]$ generated by $n+1$ forms, where at least one of the generators is a quadric, is level.
\end{corollary}
\begin{proof}
Fix some degrees $\boldsymbol{d}=(d_1,\dots,d_{n+1})$ with $d_n=2$. Then, by identifying a homogeneous polynomial with its coefficients in the monomial basis, every ideal $(f_1,\dots,f_{n+1})$ in $k[x_1,\dots,x_n]$, with $\deg f_i =d_i$, corresponds to a point in the affine space $\mathbb{A}_k^{N}$, where $N=\sum_{i=1}^{n+1}\binom{n+d_i-1}{d_i}$. It is well known, see for example Theorem 1 in \cite{Ralf-Clas-generic}, that there exists a Zariski open subset $Y\subset \mathbb{A}_k^N$ for which every ideal has the same Hilbert series, moreover this is the minimal possible Hilbert series among all ideals generated by $n+1$ forms in degrees $\boldsymbol{d}$. Note that by Theorem \ref{thm: monom-SLP}, we have that $I = (x_1^{d_1},\dots,x_{n-1}^{d_{n-1}},x_n^{2},\ell^{d_{n+1}}) \in Y$. By Theorem 6.0.8 in \cite{genericideals-howell}, there exists a non-empty open subset $X \subset Y$, for which every ideal has minimal graded Betti numbers among ideals in $Y$. Now by Corollary $\ref{cor: all is level}$ we have that $\beta^R_{n,j}(R/I) = 0$ for $j \le n+s-1$, where $s$ denote the socle degree of $R/I$ (and thus the socle degree of any ideal in $Y$). It follows that for any $J \in X$ we have $\beta^R_{n,j}(R/J)=0$ for $j \le n+s-1$, and hence $J$ is level. Since $X$ is non-empty and (Zariski) open in $\mathbb{A}_k^N$, this gives the claim.
\end{proof}

A natural question is to what 
extent these methods can be generalized further, in order to determine the Betti numbers when no generator is a quadric. That is, for which integers $d_1,\dots,d_{n+1}$ is it that $R/I$ is liftable (from some $R/(x_i^{d_i})$) to $R$? In the cases we have studied the Betti numbers reduce to a sum of two Betti numbers from another ideal. Computer calculations, done in Macaulay2, seem to indicate that such a simple relation can not be expected in general. Below we provide an example aimed to showcase this.

\begin{example}
The Betti numbers of $I=(x_1^3,x_2^3,x_3^3,x_4^3,\ell^3)$ calculated by Macaulay2 \cite{M2}, are given below.
$$
\begin{matrix}
        & 0 & 1 & 2 & 3 & 4\\
       \text{total:} & 1 & 5 & 17 & 20 & 7\\
       0: & 1 & . & . & . & .\\
       1: & . & . & . & . & .\\
       2: & . & 5 & . & . & .\\
       3: & . & . & . & . & .\\
       4: & . & . & 16 & 10 & 1\\
       5: & . & . & 1 & 10 & 6 
        \end{matrix} $$
One can check that there exists no cyclic graded $R$-module $M$ such that $\beta_{i,j}^{R}(R/I) = \beta_{i,j}^R(M) + \beta^R_{i-1,j-d}(M)$.
\end{example}

\section{SLP for the linked Artinian Gorenstein algebra}
The main goal of this short section is to show that the AG-algebra $G =(x_1^{d_1},\dots,x_{n}^{d_n})\colon (\ell^{d_{n+1}})$ enjoys the strong Lefschetz property. In fact we show the stronger statement 

\begin{theorem}\label{thm: SLP-colon}
    Let $A$ be a graded Artinian $k$-algebra with strong Lefschetz element $\ell \in A_1$. Then $A/[0:(\ell^j)]$ has the SLP, with $\ell$ as a strong Lefschetz element, for all $j> 0$.
\end{theorem}

We want to emphasize that in some ways this result is not new. Namely in Theorem 3.8 of \cite{Watanabe} a similar result is presented, however from the way it is stated it is not immediates that our statement above holds. Similarly in Theorem 4.14 of \cite{Iarrobinoa-Kanev-Book}, which is a corollary of the result from \cite{Watanabe}, it is stated that for a generic $F$, we have that  $A/[0:(F)]$ has the SLP. Also here it is not clear that the choice of a power of a Lefschetz element is ``generic enough". If one however more carefully reads the proof in \cite{Watanabe} it becomes clear that this is the case. We believe that the statement of Theorem \ref{thm: SLP-colon} is not well known, for example in \cite{diethorn2025studyquadraticcompleteintersection} when trying to prove that $G = (x_1^{2},\dots,x_{n}^{2})\colon (\ell^{{2}})$ has the SLP (their Proposition 5.5)  they state that the result of \cite{Iarrobinoa-Kanev-Book} (they cite \cite{res-n+1-general-forms}, but their arugment cites \cite{Iarrobinoa-Kanev-Book}) is not obviously applicable to this case and hence resort to other methods. For this reason we provide a simple proof, which closely follows the original idea of \cite{Watanabe}.

\begin{proof}[Proof of Theorem \ref{thm: SLP-colon}]
    Let $m$ be the smallest integer for which $$\cdot l^j \colon A_m \to A_{m+j}$$ is surjective. Then $(A/[0:(l^j)])_i = A_i$ for $i<m$ and $(A/[0:(l^j)])_i \cong A_{i+j} = l^jA_i$ for $i \ge m$. Hence the multiplication map $$(A/[0:(l^j)])_{i} \xrightarrow{\cdot l^k} (A/[0:(l^j)])_{i+k}$$ is simply, with the above identification, equal to either $\ell^k$ or a composition of $l^j$ and $l^k$ on $A$, which both clearly have maximal rank.
\end{proof}

\begin{corollary}
    Let $G  =(x_1^{d_1},\dots,x_{n}^{d_n})\colon (\ell^{d_{n+1}})$. Then $R/G$ has the SLP.
\end{corollary}
\begin{proof}
    Since $J =  (x_1^{d_1},\dots,x_{n}^{d_n})$ has the strong Lefschetz property with $\ell = x_1+\dots+x_n$, it follows from Theorem \ref{thm: SLP-colon} that $(R/J)/[0:(l^{d_{n+1}})] \cong R/[J:(l^{d_{n+1}})]$ has the SLP.
    \end{proof}

Let $S = k[X_1,\dots,X_n]$. We consider $S$ as the inverse system of $R$, meaning that we view it as an $R$-module with action given by $x_i \circ f = \frac{\partial}{\partial x_i}f$. For any homogeneous ideal $I$ we associate the $R$-module $I^\perp$ defined as
$$
I^\perp = \{f\in S\mid g\circ f = 0, \ \forall\ g \in I\}
$$ which we call the (Macaulay) inverse system of $I$.

Similarly for any graded $R$-module $B \subset S$ we can associate an ideal $$\text{Ann}(B) = \{f \in R \mid f \circ b = 0 , \ \forall \ b \in B\},$$

which we call the annihilator ideal of $B$.
This is a special case of Matlis duality (see section 3.6 in \cite{BH-CMR}) and it is well known that this gives an order reversing bijection between graded Artinian ideals of $R$ and finitely generated graded $R$-submodules of $S$. Moreover  $R/I$ is Artinian Gorenstein if and only if $I^\perp$ is a cyclic $R$-module, i.e. if $I$ is exactly the elements $a \in R$ satisfying $a\circ F = 0$ for some homogeneous polynomial $F \in S$. In this case we call $F$ the dual generator of $I$.
In general it is not easy to see from a dual generator whether the algebra has SLP or not, and hence describing such generators is an interesting problem in itself. As it turns out, the dual generators of the ideals $G$ we have been studying are easy to describe, and so from this description we obtain a new family of polynomials whose associated $AG$-algebra have the SLP.  

A more restricted version of the following lemma can be found in \cite{dual-gens-Frob} (Lemma 2.7), but the proof is essentially the same.

\begin{lemma}\label{lemma: inverse-gen} Let $R/J$ be an Artinian Gorenstein algebra with dual generator $G$.
For any Artinian Gorenstein algebra of the form $$A =R/ [J \colon (f)],$$ the dual generator of $A$ is given by
$$
F = f\circ G.
$$
\end{lemma}
\begin{proof}
 An element $g \in R$ satisfying $g \circ (f\circ G)=0$ is equivalent to 
 \begin{align*}
 (gf)\circ G = 0 \iff gf \in J \iff \\ g \in  J \colon (f) = A.
 \end{align*}
  We conclude that $\text{Ann}(F) = J:(f)$.
\end{proof}

Combining this with Theorem \ref{thm: SLP-colon} we obtain the following proposition.

\begin{prop}\label{prop: colon-SLP-dual}
    Let $F$ be a homogeneous form such that $R/\text{Ann}(F)$ has the SLP with $\ell$ as strong Lefschetz element. Then, for all $d>0$, the Artinian Gorenstein algebra given by the dual generator $\ell^d \circ F $ has the SLP.
\end{prop}
\begin{proof}
    Setting $A=R/\text{Ann}(F)$ we have from Theorem \ref{thm: SLP-colon} that $$A/[0\colon(\ell^d)] \cong R/[\text{Ann}(F):(l^d)]$$ has the SLP. Lemma \ref{lemma: inverse-gen} gives that the dual generator of $R/[\text{Ann}(F):(\ell^d)] $ is $\ell^d\circ F$, which concludes the proof.
\end{proof}

We finish this section by applying Proposition \ref{prop: colon-SLP-dual} to the monomial complete intersections, from which we obtain some new examples of homogeneous forms which are dual generators of algebras with the SLP.

\begin{corollary}\label{cor: elementary-symmetric}
For any monomial $M$ and integer $d$, the AG-algebra defined by the dual generator $(x_1+\dots+x_n)^d \circ M$ has the SLP. Specifically, the annihilator ideal of an elementary symmetric polynomial $e_{d}(x_1,\dots,x_n)$ has the SLP.
\end{corollary}
\begin{proof}
    Since the dual generator of $J = (x_1^{d_1},\dots,x_n^{d_n})$ is $x_1^{d_1-1}\cdots x_n^{d_n-1}$, the first statement follows directly from Proposition \ref{prop: colon-SLP-dual}. Noting that $(x_1+\dots+x_n)^d \circ (x_1\cdots x_n) = d!e_{n-d}(x_1,\dots,x_n)$, where $e_i$ is the i:th elementary symmetric polynomial, finishes the proof
\end{proof}

\section{Generators of $G$}
A related interesting question is to determine the generators of $G =  (x_1^{d_1},\dots,x_n^{d_n})\colon (l^{d_{n+1}})$. Although we, by Theorem \ref{thm: sum of betti}, understand the Betti numbers of these ideals when $d_n=2$ we have not been able to describe the generators completely. For the case $  (x_1^{2},\dots,x_n^{2})\colon (l^{d})$, which by Corollary \ref{cor: elementary-symmetric} is the annihilator ideal of  some elementary symmetric polynomial, we are however able to generalize the result of \cite{diethorn2025studyquadraticcompleteintersection}.
A related result, given in \cite{Waringrank-Symmetric}, shows that the annihilator ideal of a complete homogeneous symmetric polynomial satisfies the SLP, and they also determine a set of generators for these ideals.

Using the standard action of the symmetric group $S_n$ on $R=k[x_1,\dots,x_n] $, we use the notation $(f)_{S_n}$ to denote $(\{ \sigma(f) \mid \sigma \in S_n \})$.

\begin{prop}
Let $G = (x_1^{2},\dots,x_n^{2})\colon (\ell^{d}) $ and $J = (x_1^2,\dots,x_n^2)$, with $d \le n$. The generators of $G$, or equivalently of the annihilator ideal of $e_{n-d}(x_1,\dots,x_n)$, are given by
\begin{align*}
    &J + \left((x_1-x_2)\cdots (x_{n-d} - x_{n-d+1})\right)_{S_n} \quad &\text{if $n+d$ is odd,} \\
    &J + ((x_1-x_2)\cdots(x_{n-d-1}-x_{n-d})x_{n-d+1})_{S_n} \quad &\text{if $n+d$ is even.}
\end{align*}
\end{prop}
\begin{proof}
First we claim that the only minimal generators of $G$ are in degree $2$ and $l+1$, where $l=\left\lfloor \frac{n-d}{2}\right\rfloor$. If $n+d$ is even then using Proposition \ref{prop: G-res-koszul} we need only to check that there are no $n-1$:th Koszul relations of $J$ in degree $l + n-1$ or lower. But any such relation has degree equal to the sum of some $n-1$ generators of $J$, 
which always has degree $2(n-1)$, and since $l<n-1$, the conclusion follows. If $n+d$ is odd, then by Theorem $\ref{thm: sum of betti}$, $G$ is generated in degree $2$ and $\frac{n-1-d}{2} + 1 = \lfloor\frac{n-d}{2} \rfloor+1 = l+1$.

Let $G'$ denote the claimed generators.
Note that modulo $J$ we have
$$
 (x_1+\dots + x_n)\prod_{i=1}^{\frac{n-d+1}{2}}(x_{2i-1}-x_{2i})
 = (x_{n-d+2} + \dots +x_n)\prod_{i=1}^{\frac{n-d+1}{2}}(x_{2i-1}-x_{2i}).
$$ Hence 
$$
 (x_1+\dots + x_n)^d\prod_{i=1}^{\frac{n-d+1}{2}}(x_{2i-1}-x_{2i})
 = (x_{n-d+2} + \dots +x_n)^d\prod_{i=1}^{\frac{n-d+1}{2}}(x_{2i-1}-x_{2i}),
$$ and since the last factor contains $d-1$ variables it follows by the pigeonhole principle that this is 0 mod $J$. A similar argument shows the inclusion when $n+d$ is even, and using the $S_n$-symmetry of $G$ we find that $G' \subseteq G$.

Let $\textbf{in'}(G') = J+ \{ x_{i_1}\cdots x_{i_{l+1}} \mid i_1<\dots <i_{l+1} , \quad  i_j\le d+2(j-1) \}$. We first show that $\textbf{in'}(G') \subseteq \textbf{in}(G')$, where $\textbf{in}(-)$ denotes the initial ideal with respect to any monomial ordering where $x_1>x_2 > \dots >x_n$.

We begin with the case when $n+d$ is odd. It is sufficient to for each $i_j$ find a unique larger index $c_j$ to pair it with, since the product of $(x_{i_j}-x_{c_j})$ gives an element in $G'$, and the factor $(x_{i_j}-x_{c_j})$ will contribute with the factor $x_{i_j}$ to the initial term. That this is possible follows by a simple induction argument: First, as $x_{i_{l+1}} \le d+2l = n-1$ we may always pair this element with $x_n$. Inductively we may assume that we have already paired the elements $x_{i_{j+1}} \dots, x_{i_{l+1}}$ and want to show that there still exists a index larger than $i_j$ which is available. Since $i_j \le  d+2(j-1)$, we have at least $n - (d+2(j-1)) = (n-d-1) - 2j + 3 = 2(l-j) +3$ possible indices/variables to pair $i_j$ with (using that $2l = n+d-1$, as $n+d$ is odd). The previous $l-j+1$ elements $x_{i_{j+1}},\dots,x_{i_{l+1}}$ only occupy $2(i-j)+2$ of these, so there is at least one left to pair $i_j$ with. 

When $n+d$ is odd we instead have $i_{l+1} \le n$, and we may obtain this factor by simply choosing $x_{i_{l+1}}$ as the last factor of our form in $G'$. By induction we may assume that we have already paired the elements $x_{i_{j+1}}\dots,x_{l}$, along with $x_{i_{l+1}}$. There are at least $n-(d+2(j-1)) = 2(l-j)+2$ possible elements to pair $x_{i_j}$ with, and we have only used $2(l-j)+1$ for the previous elements. It follows that $\textbf{in'}(G') \subseteq \textbf{in}(G')$. 

For every square-free monomial of degree $l+1$ we can bijectively associate a lattice path from $(0,0)$ to $(n,2(l+1)-n)$ with steps $(1,1)$ and $(1,-1)$, by letting the indices of the variables appearing in the monomial $m$ correspond to the $(1,1)$ steps, with the remaining steps being set to $(1,-1)$. Let $w$ be the vector, in $\{\pm 1\}^n$, whose coordinates correspond to the second coordinate of the lattice steps described above, so that  $w_j = 1$ if and only if $x_j \mid m$. Then we claim that the criterion
$$
S_j = \sum_{i=1}^{j}w_i \ge 1-d
$$ for all $j\le n$, is equivalent to the restriction on our set of square-free monomials in $\textbf{in'}(G')$. We note that  $\sum_{i=1}^{i_j}w_i = j - (i_j-j) \ge  2-d$ if and only if $i_j \le d+2(j-1)$, and since the minimum value must be attained exactly before a positive step this is equivalent to saying that $S_j \ge 1-d$ for all $j\le n$. Hence our set of monomials is in bijection with the subset of lattice paths that never pass below the line $y=1-d$. By the reflection principle (see Theorem 10.3.1 in \cite{Krattenthaler}) the number of paths that pass below this line is the same as the number of total paths starting in $(0,-2d)$, ending in $(n,2(l+1)-n)$, which corresponds to exactly $l+1+d$ positive steps, and hence there are exactly $\binom{n}{l+1+d}$ such paths. We conclude that the number of square-free generators of $\textbf{in'}(G')$ equals
$$
\binom{n}{l+1} - \binom{n}{l+1+d}.
$$

Now from the inclusions $\textbf{in'}(G') \subseteq \textbf{in}(G') $ and $G' \subseteq G$ we have the inequalities
$$
\HF_{R/G}(i) \le \HF_{R/G'}(i) = \HF_{R/\textbf{in}(G')}(i) \le  \HF_{R/\textbf{in'}(G')}(i).
$$

We have $\HF_{R/G}(l+1) = \HF_{R/J}(l+1+d) = \binom{n}{l+1+d}$ and also $\HF_{R/\textbf{in'}(G)}= \HF_{R/J}(l+1) - \binom{n}{l+1} + \binom{n}{l+1+d} = \binom{n}{l+1+d}$ and hence the inequality is an equality in degree $l+1$.
Since $G'\subseteq G$ and they are generated in the same degrees and have the same Hilbert function in these degrees, we conclude that $G' = G$.
\end{proof}

We remark that the same type of lattice paths which we make use of in this proof also appear in \cite{exterior-quadratic}, in the context of ideals in the exterior algebra, which similarly to our case is square-free.

\section*{Acknowledgements}
I want to thank my advisor Samuel Lundqvist for his guidance, many helpful discussions, and for much help with the exposition of the paper. I also want to thank Luís Duarte for providing valuable comments on an earlier draft of this paper. Lastly, computer calculations done in Macaulay2 \cite{M2} have been crucial for understanding the objects treated in this work.
The author is supported by the Swedish Research Council VR2022-04009.

\bibliography{refs}
\bibliographystyle{alphaurl}

\end{document}